\documentclass[a4paper, 11pt]{article}
\usepackage{amsmath,amssymb,esint,amscd,xspace,fancyhdr,color,authblk,srcltx,fontenc,bbm}
\usepackage{hyperref}
\usepackage{cite}

\newtheorem{theorem}{Theorem}[section]

\newtheorem{definition}[theorem]{Definition}

\newtheorem{remark}[theorem]{Remark}
\newenvironment{proof}[1][Proof]{\textbf{#1.} }{\hfill\rule{0.5em}{0.5em}}
{\catcode`\@=11\global\let\AddToReset=\@addtoreset
\AddToReset{equation}{section}

\AddToReset{theorem}{section}

\title{A global fractional Caccioppoli-type estimate for solutions to nonlinear elliptic problems with measure data}

\author{Minh-Phuong Tran\thanks{Applied Analysis Research Group, Faculty of Mathematics and Statistics, Ton Duc Thang University, Ho Chi Minh city, Vietnam; \texttt{tranminhphuong@tdtu.edu.vn}} , Thanh-Nhan Nguyen\thanks{Department of Mathematics, Ho Chi Minh City University of Education, Ho Chi Minh city, Vietnam; \texttt{nhannt@hcmue.edu.vn}}}

\date{\today} 

\begin{document}
\maketitle

\begin{abstract}
We prove a global fractional differentiability result via the fractional Caccioppoli-type estimate for solutions to nonlinear elliptic problems with measure data. This work is in fact inspired by the recent paper [B. Avelin, T. Kuusi, G. Mingione, {\em Nonlinear Calder\'on-Zygmund theory in the limiting case}, Arch. Rational. Mech. Anal. {\bf 227}(2018), 663--714], that was devoted to the local fractional regularity for the solutions to nonlinear elliptic equations with right-hand side measure, of type $-\mathrm{div}\, \mathcal{A}(\nabla u) = \mu$ in the limiting case. Being a contribution to recent results of identifying function classes that solutions to such problems could be defined, our aim in this work is to establish a global regularity result in a setting of weighted fractional Sobolev spaces, where the weights are powers of the distance function to the boundary of the smooth domains.

\medskip

\medskip


\end{abstract}   
  
\section{Introduction and main results}
\label{sec:intro}
In this study, we are interested in the following Dirichlet problem with measure data
\begin{equation}\label{eq:elliptictype}
\begin{cases}
-\mathrm{div} \, \mathcal{A}(x,\nabla u) &= \ \mu \qquad \text{in} \ \ \Omega, \\
\hspace{1.5cm} u &= \ 0 \qquad \text{on} \ \ \partial \Omega.
\end{cases}
\end{equation}
Here, $\Omega$ is an open bounded domain of $\mathbb{R}^n$ ($n \ge 2$); data $\mu$ is a Borel measure with finite mass in $\Omega$; the nonlinearity $\mathcal{A}$ is a Carath\'eodory vector field defined on $\Omega \times \mathbb{R}^n$ and $\mathcal{A}$ satisfies the following ellipticity and growth assumptions 
\begin{equation}\label{eq:cond-ellipgrow}
\begin{cases}
|\mathcal{A}(x,z)| + |\partial_z \mathcal{A}(x,z)|(|z|^2+\kappa^2) \le c_{\mathcal{A}}(|z|^2+\kappa^2)^{\frac{p-1}{2}},\\
c_{\mathcal{A}}^{-1}(|z|^2+\kappa^2)^{\frac{p-2}{2}}|\zeta|^2 \le \langle \partial_z \mathcal{A}(x,z)\zeta,\zeta \rangle,
\end{cases}
\end{equation}
for every $z, \zeta \in \mathbb{R}^n$ and $x \in \Omega$. Moreover, it is important to remark that in the above assumptions of $\mathcal{A}$ as follows: we here only consider $p>2-\frac{1}{n}$, and $c_{\mathcal{A}}>1$ is the given ellipticity constant, $\kappa \in [0,1]$ represents the degeneracy parameter to distinguish between two cases of problems in our study: $\kappa=0$ for the degenerate case and $\kappa>0$ the non-degenerate case respectively. Moreover, when $2-\frac{1}{n}< p<2$, we further have the imposition of symmetry condition for the operator $\mathcal{A}$:
\begin{align}\label{eq:cond-sym}
\partial_z \mathcal{A}  \ \ \text{is symmetric} \Leftrightarrow \partial_{z_j} \mathcal{A}_i = \partial_{z_i} \mathcal{A}_j, \quad \forall i, j \in \{1,2,...,n\}.
\end{align}

A significant case of $\mathcal{A}$ is the $p$-Laplace operator $\Delta_p u = |\nabla u|^{p-2}\nabla u$. Although we consider the general degenerate equations of the type in~\eqref{eq:elliptictype}, regularity results in this work can be straightforward in the $p$-Laplace equations.

In the past years, a great deal of effort has gone into investigating nonlinear elliptic/parabolic equations involving measure data. Apart from theoretical interest in mathematics, these equations have also entered in several models describing numerous phenomena in the applied sciences for instance non-Newtonian fluids, electrorheological fluids, flows in porous medium, dislocation and image restoration problems, etc. Together with the researches on the existence and uniqueness of solutions to~\eqref{eq:elliptictype}, the question of regularity concerning the integrability and differentiability properties of solutions also obtained a lot of attraction. From the fact that when $p=2$, the equation $-\Delta u=-\mathrm{div}(\nabla u)=\mu$ attains a certain result 
\begin{align*}
\mu \in L^q_{\mathrm{loc}}(\Omega) \Longrightarrow \nabla u \in W^{1,q}_{\mathrm{loc}}(\Omega), \quad 1<q<\infty,
\end{align*}
and that no longer true for $q=1$, the fractional Sobolev spaces were studied to get the maximal regularity estimates. For instance, in the recent paper by Avelin \textit{et al.} in~\cite{AKM18}, ones can prove that
\begin{align}\label{eq:Lap}
\mu \in L^1_{\mathrm{loc}}(\Omega) \Longrightarrow \nabla u \in W^{\sigma,1}_{\mathrm{loc}}(\Omega), \quad 0<\sigma<1.
\end{align}
Moreover, also in the same paper, authors proved  a very important regularity result for local SOLA solutions to problem~\eqref{eq:elliptictype} when $2 - \frac{1}{n} <p \le 2$, that can be re-stated in the following theorem for readers' convenience:
\begin{theorem} [Limiting case of Calder\'on-Zygmund theory,~\cite{AKM18}]
\label{theo:Avelin}
Let $p > 2 - \frac{1}{n}$ and $\Omega$ be an open subset in $\mathbb{R}^n$. Assume that the operator $\mathcal{A}$ satisfies assumptions~\eqref{eq:cond-ellipgrow}-\eqref{eq:cond-sym} and $u \in W^{1,\max\{1,p-1\}}_{\mathrm{loc}}(\Omega)$ is a SOLA solution to~\eqref{eq:elliptictype}. Then for any $\sigma \in (0,1)$ one has
\begin{align}\label{eq:A_inW}
\mathcal{A}(\nabla u) \in W^{\sigma,1}_{\mathrm{loc}}(\Omega).
\end{align}
Moreover, one can find a constant $C = C(c_{\mathcal{A}},\sigma,n,p)>0$ such that 
\begin{align}
\label{eq:diff_est_local}
\fint_{B_{R/2}}{\int_{B_{R/2}}{\frac{|\mathcal{A}(\nabla u (x)) - \mathcal{A}(\nabla u (y))|}{|x-y|^{n+\sigma}}}dxdy} \le \frac{C}{R^\sigma}\fint_{B_R}{|\mathcal{A}(\nabla u(x))|dx} + \frac{C}{R^\sigma}\left[\frac{|\mu|(B_R)}{R^{n-1}} \right],
\end{align}
for every ball $B_R \Subset \Omega$.
\end{theorem}

It is worth noting that for the nonlinear elliptic problem with measure data, the weak solutions may be not unique (see a counterexample in~\cite{Serrin}). Therefore, a rather extensive literature is in place regarding to many definitions of solutions to such equations, where the existence and uniqueness are possible: entropy solutions in~\cite{BBGGPV1995}; renormalized solutions in~\cite{Maso1999, Maso1997}, SOLA in~\cite{BG1989, BG1992} (see Definition~\ref{def:distri_sol}). Here, in this article we shall adopt the concept of SOLA solutions when dealing with measure data problem~\eqref{eq:elliptictype}, whose definition will be specified in Section~\ref{sec:SOLA} below. 

In the case when $p>2 - \frac{1}{n}$, an impressive result $\mathcal{A}(\nabla u) \in W^{\sigma,1}_{\mathrm{loc}}(\Omega)$ comes along with~\eqref{eq:diff_est_local} in Theorem~\ref{theo:Avelin} complete linearization effect of the equation with respect to fractional differentiability of weak solutions. In addition, one concludes that for the nonlinear problem~\eqref{eq:elliptictype}, results obtained are exactly the same as the linear case $-\Delta u = \mu$ via fundamental solutions as in~\eqref{eq:Lap}. Besides, there have been a number of intensive contributions for local fractional regularity for the solutions to measure data problem, such as the differentiability for $\nabla u$ when $p=2$ in~\cite{Min2007}; the Calder\'on-Zygmund type estimates for $\mathcal{A}(\nabla u)$ in the scale of Besov or Triebel-Lizorkin spaces when $p\ge 2, n=2$ obtained in~\cite{BDW2020}; results for vectorial case in~\cite{KM2018}; or some results concerning the global gradient estimates in Lorentz or Morrey spaces in~\cite{MPT2018,MPT19,PNCCM,HP2019} with the singular range of $p$, i.e. $1 < p \le 2- \frac{1}{n}$; and many many further interesting results in~\cite{DHM1997,CM2014,BCDK2018,DM2010,HH2017,Veron,PNJDE}, etc. 

Motivated by the work in~\cite{AKM18}, in the present article we concentrate in studying the limiting case of Calder\'on-Zygmund estimates in Theorem~\ref{theo:Avelin} up to the boundary with smoothness assumption on $\partial\Omega$. Restrict to the case $p>2-\frac{1}{n}$, we herein construct an appropriate function class to achieve a global regularity results that corresponding to the ones proved in~\cite[Theorem 1.2]{AKM18}. More precisely, based on the idea to work in weighted fractional Sobolev spaces equipped with weights chosen as a power of the distance to a point at the boundary (we refer to Definition~\ref{def:Gagli-w} below), this leads us to apply local results in Theorem~\ref{theo:Avelin} to a set of sufficiently small balls in $\Omega$. 

We now state our main results via two following theorems.
\begin{theorem}\label{theo:main}
Let $p > 2 - \frac{1}{n}$, $\sigma \in (0,1)$ and $\Omega$ be an open bounded and smooth domain in $\mathbb{R}^n$. Assume that the operator $\mathcal{A}$ satisfies assumptions~\eqref{eq:cond-ellipgrow}-\eqref{eq:cond-sym} and $u \in W^{1,\max\{1,p-1\}}(\Omega)$ is a SOLA solution to~\eqref{eq:elliptictype}. Then for every $\alpha$, $\beta>0$ satisfying $\alpha+\beta>\sigma$, one has
\begin{align}\label{eq:main}
\mathcal{A}(\nabla u) \in W^{\sigma,1}_{G}(\Omega;\alpha,\beta).
\end{align}
Moreover, one can find a constant $C = C(c_{\mathcal{A}},\sigma,n,p,\alpha,\beta,\mathrm{diam}(\Omega))>0$ such that
\begin{align}\label{eq:est-main} 
\int_{\Omega}{\int_{\Omega} d^{\alpha}(x) d^{\beta}(y) {\frac{|\mathcal{A}(\nabla u (x)) - \mathcal{A}(\nabla u (y))|}{|x-y|^{n+\sigma}}}dxdy} \le C \left(\int_{\Omega} |\mathcal{A}(\nabla u(x))| dx + |\mu|(\Omega)\right),
\end{align}
where $d^{\vartheta}(x) := \left[\mathrm{dist}(x,\partial \Omega)\right]^{\vartheta}$.
\end{theorem}

\begin{theorem}\label{theo:main-B}
Let $p \ge 2$, $\sigma \in (0,1)$ and $\Omega$ be an open bounded and smooth domain in $\mathbb{R}^n$. Assume that the operator $\mathcal{A}$ satisfies assumptions~\eqref{eq:cond-ellipgrow} and $u \in W^{1,p-1}(\Omega)$ is a SOLA solution to~\eqref{eq:elliptictype}. Then for every $\alpha$, $\beta>0$ satisfying $\alpha+\beta>\sigma$ and $\frac{1}{p-1} \le \gamma \le 1$ one has
$\mathbb{E}(\nabla u) \in W^{\gamma\sigma,\frac{1}{\gamma}}_{G}(\Omega;\alpha,\beta)$,
where the function $\mathbb{E}: \ \mathbb{R}^n \mapsto [0,\infty)$ defined by
\begin{align}\label{def:E}
\mathbb{E}(\xi) = \left(|\xi| + \kappa\right)^{\gamma p - \gamma-1} \xi, \qquad \xi \in \mathbb{R}^n.
\end{align}
In particular, one can find a constant $C = C(c_{\mathcal{A}},\sigma,n,p,\alpha,\beta,\gamma)>0$ such that 
\begin{align}\label{est-main-B}
[\mathbb{E}(\nabla u)]_{W_G^{\gamma\sigma,\frac{1}{\gamma}}(\Omega;\alpha,\beta)} \le C \left(\int_{\Omega} |\mathcal{A}(\nabla u(x))| dx + |\mu|(\Omega)\right)^{\gamma}.
\end{align}
\end{theorem}

It is remarkable that when $\gamma=1$, one easily obtain~\eqref{eq:est-main} from~\eqref{est-main-B}. Moreover, in the case $\gamma = \frac{1}{p-1}$, one has
\begin{align}
\nabla u \in W_G^{\frac{\sigma}{p-1},p-1}(\Omega;\alpha,\beta), \quad \mbox{ for every } \sigma \in (0,1).
\end{align}

The remainder of our paper is organized as follows. In the next section we introduce some mathematical preliminaries and function spaces. This section focuses on the concept of weighted fractional Sobolev spaces by introducing some basic notation, definitions and some properties of function spaces.  Then, we end up with a section devoted to proving main results in this paper, and it allows us to conclude a global fractional Caccioppoli type inequality for solutions to measure data problem~\eqref{eq:elliptictype}.

\section{Preliminaries and function spaces}\label{sec:pre}
\subsection{Basic notation}

The general constant, although in various occurrences from line to line, will always be denoted by $C$. And the dependencies of $C$ will be highlighted between parentheses if needed. For example, when $C$ depends on some real numbers $n,p,\sigma,c_\mathcal{A}$ we will write $C = C(n,p,\sigma,c_\mathcal{A})$. In what follows, we simply write $B_{\varrho}(\xi)$ to denote the ball with radius $\varrho$ and centered at $\xi \in \Omega$; and with a not relevant center it will be written $B_{\varrho}$ for simplicity. Throughout the paper, for $1 \le q<\infty$, we employ the familiar notation $L^q (\Omega)$ to denote the usual Lebesgue spaces; and $W^{s,q}(\Omega)$ stands for the Sobolev spaces. Finally, for a given measurable subset $\mathcal{O} \subset \mathbb{R}^n$, the standard notation of integral average of a function $\varphi \in L^1(\mathcal{O})$ is denoted by
\begin{align*}
(\varphi)_{\mathcal{O}} = \fint_{\mathcal{O}}{\varphi(\xi)d\xi} = \frac{1}{|\mathcal{O}|}\int_{\mathcal{O}}{\varphi(\xi)d\xi},
\end{align*}
where $|\mathcal{O}|$ stands for the Lebesgue measure of $\mathcal{O}$ in $\mathbb{R}^n$.

\subsection{The notion of solution: SOLA}\label{sec:SOLA}

In a natural way, one has the distribution notion of solutions to~\eqref{eq:elliptictype} as in the next definition.

\begin{definition}[Distributional solution]\label{def:distri_sol}
A function $u \in W^{1,1}_0(\Omega)$ is said to be a very weak solution to~\eqref{eq:elliptictype} if $\mathcal{A}(x,\nabla u) \in L^1(\Omega;\mathbb{R}^n)$ and
\begin{align}\nonumber
\int_\Omega{\langle \mathcal{A}(x,\nabla u),\nabla \varphi\rangle dx} = \int_\Omega{\varphi d\mu},
\end{align}
holds for all $\varphi \in C_c^\infty(\Omega)$.
\end{definition}

This type of distributional solutions may exist. However, according to the counterexample by Serrin in~\cite{Serrin}, the problem of uniqueness of such solutions poses. For this reason, many possible definitions have been proposed such as the notions of entropy solutions, SOLA - \emph{Solutions Obtained as Limits of Approximations}, renormalized solutions, etc (see~\cite{BG1989, BG1992,Dall1996, Min2007} and many references therein or later concerning the nonlinear measure data problems). For the present purpose of this paper, we confine ourselves to the notion of SOLA, whose definition can be figured below.

\begin{definition}[Local SOLA,~\cite{BG1989,BG1992,AKM18}]\label{def:SOLA}
A function $u \in W^{1,1}_{\mathrm{loc}}(\Omega)$ is a local SOLA to~\eqref{eq:elliptictype} under assumptions~\eqref{eq:cond-ellipgrow} if one can find a sequence $\{u_k\} \subset W^{1,p}_{\mathrm{loc}}(\Omega)$ to the following equations
\begin{align*}
-\mathrm{div}(\mathcal{A}(x,\nabla u_k)) = \mu_k \in L^\infty_{\mathrm{loc}}(\Omega)
\end{align*}
such that $u_k \rightharpoonup u$ weakly in $W^{1,1}_{\mathrm{loc}}(\Omega)$, where the sequence $\{\mu_k\}$ converges weakly in the sense of measures and satisfies
\begin{align*}
\limsup_k{|\mu_k|(B)} \le |\mu|(B),
\end{align*}
for any ball $B\Subset \Omega$.
\end{definition}

\begin{remark}\label{rem:1}
Follow the arguements in~\cite{BG1989} and~\cite[Proposition 2.2]{AKM18}, with $p>2-\frac{1}{n}$, if $u \in W^{1,1}_{\mathrm{loc}}(\Omega)$ is a local SOLA to problem~\eqref{eq:elliptictype} and $\{u_k\}$ is the approximating solutions in Definition~\ref{def:SOLA}, then $\{u_k\}$ converges strongly to $u$ in $W^{1,q}_{\mathrm{loc}}(\Omega)$ for any $q<\min\left\{p, \frac{n(p-1)}{n-1}\right\}$. In particular, this leads to establish that there exists a very weak solution $u \in W^{1,\max\{1,p-1\}}_{\mathrm{loc}}(\Omega)$ to~\eqref{eq:elliptictype} .
\end{remark}

\subsection{Function spaces}
\label{sec:FS}

Let us firstly recall the classical definition of fractional Sobolev spaces in the sense of Gagliardo, see~\cite{DPV12, AKM18}. 

\begin{definition}{(Gagliardo's fractional Sobolev space)}\label{def:Gagliardosp}
Let $\Omega$ be a general open set in $\mathbb{R}^n$ with $n \ge 2$ and a fractional exponent $s \in (0,1)$. Then, for any $1 \le q < \infty$, the fractional Sobolev space $W^{s,q}_G(\Omega)$ is defined by
\begin{align}
W_G^{s,q}(\Omega) = \left\{v \in L^q(\Omega): \frac{|v(x) - v(y)|}{|x-y|^{\frac{n}{q}+s}} \in L^q(\Omega \times \Omega) \right\},
\end{align}
and this is the Banach space endowed with the Gagliardo type norm as below
\begin{align}\label{eq:Gnorm}
\|v\|_{W_G^{s,q}(\Omega)} = \left[ \int_\Omega{|v(x)|^q dx} + \int_{\Omega}{\int_{\Omega}{\frac{|v(x)-v(y)|^q}{|x-y|^{n+sq}}}dxdy} \right]^{\frac{1}{q}}.
\end{align}
\end{definition}

In what follows, we will denote 
\begin{align}\label{eq:Gseminorm}
[v]_{W_G^{s,q}(\Omega)} := \left[ \int_{\Omega}{\int_{\Omega}{\frac{|v(x)-v(y)|^q}{|x-y|^{n+sq}}}dxdy} \right]^{\frac{1}{q}}
\end{align}
being the Gagliardo semi-norm of $v$. The space $W_G^{s,q}(\Omega)$ is the interpolated space between $L^q(\Omega)$ and $W^{1,q}(\Omega)$. Moreover, we write $W^{\sigma,1}_{G,\mathrm{loc}}(\Omega)$ to denote the set of all functions $v \in W^{\sigma,1}_{G}(\Omega')$ for any open subset $\Omega'$ of $\Omega$.

There are some well-known Sobolev's embedding theorems in the case of fractional spaces. For instance, one refers to~\cite[Proposition 2.1]{DPV12} for the following property
\begin{align*}
 W_G^{t,q}(\Omega) \subseteq W_G^{s,q}(\Omega), \quad \mbox{for all } \ t \in (s,1).
\end{align*}

On the other hand, with an additional assumption on the regularity on the boundary of the domain $\Omega$, particularly the bounded Lipschitz domain $\Omega$ (see~\cite[Proposition 2.2]{DPV12}), it obtains that
\begin{align*}
W_G^{1,q}(\Omega) \subseteq W_G^{s,q}(\Omega).
\end{align*}

Let us here also recall the \emph{fractional Sobolev embedding} as follows. If $\Omega \subset \mathbb{R}^n$ is a domain with $C^{0,1}$-boundary and $sq<n$, then
\begin{align*}
W^{s,q}(\Omega) \hookrightarrow L^{\frac{nq}{n-sq}}(\Omega),
\end{align*}
with continuous embedding. In other words, if $\Omega$ is a bounded Lipschitz domain and $sq<n$ then one can find $C = C(n,q,s,\mathrm{diam}(\Omega),[\partial \Omega]_{0,1})>0$ such that for $v \in L^{\frac{nq}{n-sq}}(\Omega)$, there holds
\begin{align*}
\|v\|_{L^{\frac{nq}{n-sq}}(\Omega)} \le C \|v\|_{W_G^{s,q}(\Omega)}.
\end{align*}
Moreover, a Poincar{\'e}-type inequality in fractional Sobolev spaces can be stated as below. There exists $C = C(n,q,s)>0$ such that the following Poincar\'e inequality holds
\begin{align}\label{eq:Poincare}
\int_{B_R}{|v - (v)_{B_R}|^q dx} \le C R^{sq}\int_{B_R}{\int_{B_R}{\frac{|v(x)-v(y)|^q}{|x-y|^{n+sq}}dx}dy},
\end{align}
for all $v \in W^{s,q}_G(B_R)$ and every ball $B_R \Subset \Omega$.

In this paper, for the present purpose, we will generalize and consider a weighted Gagliardo's fractional Sobolev space associated to a power of the distance to a point at boundary of the domain. More precisely, in order to study the fractional order regularity for gradient of solutions to measure data problem~\eqref{eq:elliptictype}, we employ here the \emph{weighted Gagliardo's fractional Sobolev spaces} in Definition~\ref{def:Gagli-w} below.

\begin{definition}{(A weighted Gagliardo's fractional Sobolev space)}
\label{def:Gagli-w}
Let $\Omega$ be an open bounded and Lipschitz domain in $\mathbb{R}^n$. For any $q \in [1,\infty)$, $s \in (0,1)$ and $\alpha$, $\beta \ge 0$, we define the weighted Gagliardo's fractional Sobolev space
\begin{align}
W_{G}^{s,q}(\Omega;\alpha,\beta) = \left\{v \in L^q(\Omega): \ d^{\frac{\alpha}{q}}(x)d^{\frac{\beta}{q}}(y)\frac{|v(x) - v(y)|}{|x-y|^{\frac{n}{q}+s}} \in L^q(\Omega \times \Omega) \right\},
\end{align}
which is endowed with the natural norm
\begin{align}\label{eq:G-w}
\|v\|_{W_G^{s,q}(\Omega;\alpha,\beta)} = \left[ \int_\Omega{|v(x)|^q dx} + \int_{\Omega}{\int_{\Omega}{d^{\alpha}(x)d^{\beta}(y)\frac{|v(x)-v(y)|^q}{|x-y|^{n+sq}}}dxdy} \right]^{\frac{1}{q}}.
\end{align}
Here $d(x)=\mathrm{dist}(x,\partial \Omega)$ denotes the distance from $x$ to the boundary of $\Omega$. 
\end{definition}
When $\alpha=\beta$, we simply write $W_{G}^{s,q}(\Omega;\alpha)$ instead of $W_{G}^{s,q}(\Omega;\alpha,\alpha)$. Similar to the non-weight spaces, we also denote the following term
\begin{align}\label{eq:Gsemi-w}
[v]_{W_G^{s,q}(\Omega;\alpha,\beta)} := \left[ \int_{\Omega}{\int_{\Omega}{d^{\alpha}(x)d^{\beta}(y)\frac{|v(x)-v(y)|^q}{|x-y|^{n+sq}}}dxdy} \right]^{\frac{1}{q}}
\end{align}
for the weighted Gagliardo semi-norm of $v \in W_G^{s,q}(\Omega;\alpha,\beta)$. 

On the other hand, it is clear to see that
\begin{align*}
[v]_{W_G^{s,q}(\Omega;\alpha,\beta)} \le \left(\mathrm{diam}(\Omega)\right)^{\frac{\alpha+\beta}{q}} [v]_{W_G^{s,q}(\Omega)},
\end{align*}
and this allows us to obtain the following relation holds
\begin{align*}
W_G^{s,q}(\Omega) \subset W_G^{s,q}(\Omega;\alpha,\beta).
\end{align*}

\section{Proofs of main theorems}\label{sec:proofs}

By applying the local results from Theorem~\ref{theo:Avelin} and some important properties of weighted fractional Sobolev's spaces discussed in Section~\ref{sec:pre}, we are ready to prove the main theorems in this section, where the fractional regularity for the solutions to~\eqref{eq:elliptictype} up to the boundary of a smooth domain $\Omega$ will be established in the setting of weighted fractional Sobolev spaces. 

\begin{proof}[Proof of Theorem~\ref{theo:main}]
Let us consider $0< R_0 < \mathrm{diam}(\Omega)/2$ and  
$$\Omega_0 := \left\{x \in \Omega: \  0< d(x) \le \frac{R_0}{2}  \right\}$$ 
as the set of points near the boundary of $\Omega$. We first decompose $\Omega_0 = \bigcup_{k=1}^{\infty} \Omega_k$, where $\Omega_k$ is defined by 
\begin{align*}
\Omega_k := \left\{x \in \Omega: \  r_{k+1} < d(x) \le r_k  \right\}, 
\end{align*}
with $r_k = 2^{-k} R_0$ for every $k \in \mathbb{N}^*$. For simplicity of notation, let us introduce the following function
\begin{align*}
\mathbb{T}(x,y) := d^{\alpha}(x) d^{\beta}(y) {\frac{|\mathcal{A}(\nabla u (x)) - \mathcal{A}(\nabla u (y))|}{|x-y|^{n+\sigma}}}, \quad x, \, y \in \Omega, \, x \neq y.
\end{align*}
The integral of $\mathbb{T}$ over $\Omega \times \Omega$ can be split into three terms of integrals as follows
\begin{align*}
\int_{\Omega} \int_{\Omega} \mathbb{T}(x,y) dx dy & = \int_{\Omega_0} \int_{\Omega_0} \mathbb{T}(x,y) dx dy  + 2  \int_{\Omega_0} \int_{\Omega\setminus \Omega_0} \mathbb{T}(x,y) dx dy \\
& \qquad \qquad  + \int_{\Omega\setminus \Omega_0} \int_{\Omega\setminus \Omega_0} \mathbb{T}(x,y) dx dy \\
& =: (\mathbb{I}) + 2(\mathbb{II}) + (\mathbb{III}),
\end{align*}
where 
\begin{align*}
(\mathbb{III}) &= \int_{\Omega\setminus \Omega_0} \int_{\Omega\setminus \Omega_0} \mathbb{T}(x,y) dx dy; \quad (\mathbb{II}) = \int_{\Omega_0} \int_{\Omega\setminus \Omega_0} \mathbb{T}(x,y) dx dy; 
\end{align*}
and 
\begin{align*}
(\mathbb{I}) &= \int_{\Omega_0} \int_{\Omega_0} \mathbb{T}(x,y) dx dy.
\end{align*}
One can see that the two last terms $(\mathbb{II})$ and $(\mathbb{III})$ containing the integrals over the interior domain $\Omega\setminus \Omega_0$, which can be estimated by applying the local inequality~\eqref{eq:diff_est_local} in Theorem~\ref{theo:Avelin}. Therefore, the remaining difficulty lies in the first one $(\mathbb{I})$. Let us now rewrite $(\mathbb{I})$ as 
\begin{align}\nonumber
(\mathbb{I}) & = \sum_{k,j=1}^{\infty} \int_{\Omega_{k}}\int_{\Omega_j} \mathbb{T}(x,y) dx dy \\ \nonumber
& = \sum_{|k-j|\ge 2} \int_{\Omega_{k}}\int_{\Omega_j} \mathbb{T}(x,y) dx dy + \sum_{|k-j| = 1} \int_{\Omega_{k}}\int_{\Omega_j} \mathbb{T}(x,y) dx dy \\
& \qquad \qquad + \sum_{k=1}^{\infty} \int_{\Omega_{k}}\int_{\Omega_k} \mathbb{T}(x,y) dx dy\\ \label{est:I3-a}
& =: (\mathbb{I})_1 + (\mathbb{I})_2 + (\mathbb{I})_3,
\end{align}
where 
$$(\mathbb{I})_1 = \sum_{|k-j|\ge 2} \int_{\Omega_{k}}\int_{\Omega_j} \mathbb{T}(x,y) dx dy; \quad (\mathbb{I})_2 = \sum_{|k-j| = 1} \int_{\Omega_{k}}\int_{\Omega_j} \mathbb{T}(x,y) dx dy,$$
and 
$$(\mathbb{I})_3 = \sum_{k=1}^{\infty} \int_{\Omega_{k}}\int_{\Omega_k} \mathbb{T}(x,y) dx dy.$$
We are now in order consider each term on the right-hand side of~\eqref{est:I3-a}. In the first term $(\mathbb{I})_1$, for any $x \in \Omega_k$ and $y \in \Omega_j$ with $|k - j| \ge 2$, one has 
$$|x- y| \ge \max\left\{\frac{r_k}{4}, \frac{r_j}{4}\right\} \ge \frac{r_k+r_j}{8},$$ 
and this allows us to arrive that
\begin{align}\label{est:basic}
\int_{\Omega_{k}}\int_{\Omega_j} & d^{\alpha}(x) d^{\beta}(y)  {\frac{|\mathcal{A}(\nabla u (x))|}{|x-y|^{n+\sigma}}} dx dy \notag \\
& \le r_k^{\alpha} r_j^{\beta} \int_{\Omega_{k}}  \left(\int_{\left\{|\xi|\ge \frac{r_k+r_j}{8}\right\}} {\frac{1}{|\xi|^{n+\sigma}}} d\xi\right) |\mathcal{A}(\nabla u (x))| dx \notag \\
& \le 8^{\sigma} \frac{r_k^{\alpha} r_j^{\beta}}{(r_k+r_j)^{\sigma}} \int_{\Omega_{k}}  \left(\int_{\left\{|\xi|\ge 1\right\}} {\frac{1}{|\xi|^{n+\sigma}}} d\xi\right) |\mathcal{A}(\nabla u (x))| dx \notag \\
& \le C(n,\sigma) \frac{r_k^{\alpha} r_j^{\beta}}{(r_k+r_j)^{\sigma}} \int_{\Omega_{k}}  |\mathcal{A}(\nabla u (x))| dx.
\end{align}
It is important to remark that the last inequality in~\eqref{est:basic} comes from the fact that the integral $\int_{\left\{|\xi|\ge 1\right\}} {\frac{1}{|\xi|^{n+\sigma}}} d\xi$ is finite since $n + \sigma>n$. Applying this estimate into $(\mathbb{I})_1$, one has
\begin{align}\label{est:I31-a}
(\mathbb{I})_1 & = \sum_{|k-j|\ge 2} \int_{\Omega_{k}}\int_{\Omega_j} \mathbb{T}(x,y) dx dy \notag \\ 
 & \le \sum_{|k-j|\ge 2} \left( \int_{\Omega_{k}}\int_{\Omega_j} d^{\alpha}(x) d^{\beta}(y) {\frac{|\mathcal{A}(\nabla u (x))|}{|x-y|^{n+\sigma}}} dx dy \right. \notag \\ &  \qquad \qquad \qquad \left. +  \int_{\Omega_{k}}\int_{\Omega_j} d^{\alpha}(x) d^{\beta}(y) {\frac{|\mathcal{A}(\nabla u (y))|}{|x-y|^{n+\sigma}}} dx dy \right) \notag \\  
& \le C(n,\sigma) \left( (\mathbb{I})_{11} +  (\mathbb{I})_{12} \right),
\end{align}
where 
\begin{align*}
(\mathbb{I})_{11} := \sum_{k-j \ge 2} \left( \frac{r_k^{\alpha} r_j^{\beta}}{(r_k+r_j)^{\sigma}} \int_{\Omega_{k}}  |\mathcal{A}(\nabla u (x))| dx + \frac{r_j^{\alpha} r_k^{\beta}}{(r_k+r_j)^{\sigma}} \int_{\Omega_{j}}  |\mathcal{A}(\nabla u (y))| dy\right),
\end{align*}
and 
\begin{align*}
(\mathbb{I})_{12} := \sum_{j-k \ge 2} \left( \frac{r_k^{\alpha} r_j^{\beta}}{(r_k+r_j)^{\sigma}} \int_{\Omega_{k}}  |\mathcal{A}(\nabla u (x))| dx + \frac{r_j^{\alpha} r_k^{\beta}}{(r_k+r_j)^{\sigma}} \int_{\Omega_{j}}  |\mathcal{A}(\nabla u (y))| dy\right),
\end{align*}
respectively. At this step, since $r_k \le r_j$ for all $k \ge j+2$, there holds
\begin{align*}
(\mathbb{I})_{11} & =  \sum_{j=1}^{\infty} r_j^{\beta-\sigma} \sum_{k=j+2}^{\infty} \frac{r_k^{\alpha} }{(2^{j-k}+1)^{\sigma}}\int_{\Omega_{k}}  |\mathcal{A}(\nabla u (x))| dx \\
& \qquad \qquad \qquad + \sum_{j=1}^{\infty} r_j^{\alpha-\sigma} \int_{\Omega_{j}}  |\mathcal{A}(\nabla u (y))| dy \sum_{k=j+2}^{\infty} \frac{r_k^{\beta} }{(2^{j-k}+1)^{\sigma}} \\
& \le  \sum_{j=1}^{\infty} r_j^{\alpha+\beta-\sigma} \sum_{k=j+2}^{\infty} \int_{\Omega_{k}}  |\mathcal{A}(\nabla u (x))| dx  + \sum_{j=1}^{\infty} r_j^{\beta-\sigma} \int_{\Omega_{j}}  |\mathcal{A}(\nabla u (y))| dy \sum_{k=j+2}^{\infty} r_k^{\alpha}  \\
& \le \int_{\Omega_0} |\mathcal{A}(\nabla u (x))| dx \sum_{j=1}^{\infty} r_j^{\alpha+\beta-\sigma}   + C(\alpha) \sum_{j=1}^{\infty} r_j^{\alpha+\beta-\sigma} \int_{\Omega_{j}}  |\mathcal{A}(\nabla u (y))| dy  \\
& \le C(\alpha) \int_{\Omega_0} |\mathcal{A}(\nabla u (x))| dx  \sum_{j=1}^{\infty} r_j^{\alpha+\beta-\sigma},
\end{align*}
and similarly, it yields
\begin{align*}
(\mathbb{I})_{12} & = \sum_{j-k \ge 2} \left( \frac{r_k^{\alpha} r_j^{\beta}}{(r_k+r_j)^{\sigma}} \int_{\Omega_{k}}  |\mathcal{A}(\nabla u (x))| dx + \frac{r_j^{\alpha} r_k^{\beta}}{(r_k+r_j)^{\sigma}} \int_{\Omega_{j}}  |\mathcal{A}(\nabla u (y))| dy\right) \\
& \le C(\beta) \int_{\Omega_0} |\mathcal{A}(\nabla u (x))| dx \sum_{j=1}^{\infty} r_j^{\alpha+\beta-\sigma},
\end{align*}
which can be substituted into~\eqref{est:I31-a} to reduce
\begin{align}\label{est:I31}
(\mathbb{I})_1  & \le C(n,\sigma,\alpha,\beta) \int_{\Omega_{0}} |\mathcal{A}(\nabla u(x))| dx \sum_{j=1}^{\infty} r_j^{\alpha+\beta-\sigma}.
\end{align}
To deal with the third term $(\mathbb{I})_3$, we first notice that $\Omega_k$ can be covered by $N_k \sim \frac{|\partial \Omega|}{r_k}$ balls with radius $r_k$ and centered at $z^k_{l} \in \Omega_k$, $l = \overline{1,N_k}$, that means
\begin{align*}
\Omega_k \subset \bigcup_{l=1}^{N_k} B_{r_k}(z^{k}_{l}) = \bigcup_{z^{k}_{l} \in Q_k} B_{r_k}(z^{k}_{l}),
\end{align*}
where $Q_k := \left\{z^{k}_{l} \in \Omega_k: \ l \in \{1,2,3,...,N_k\}\right\}$. By the geometric feature of each set $Q_k$, we can decompose the integral in $\Omega_k \times \Omega_k$ as follows
\begin{align}\nonumber
\int_{\Omega_k} \int_{\Omega_k} \mathbb{T}(x,y) dx dy & \le \sum_{z^{k}_i, z^{k}_j \in Q_k} \int_{B_{r_k}(z^{k}_i)} \int_{B_{r_k}(z^{k}_j)} \mathbb{T}(x,y) dx dy \\ \nonumber
& \le \sum_{z^{k}_i \in Q_k} \sum_{z^{k}_j \in Q_{k, z^{k}_i}}  \int_{B_{r_k}(z^{k}_i)} \int_{B_{r_k}(z^{k}_j)} \mathbb{T}(x,y) dx dy \\ \label{est:Om-k}
& \quad  + \sum_{z^{k}_i \in Q_k} \sum_{z^{k}_j \in Q_k \setminus Q_{k, z^{k}_i}} \int_{B_{r_k}(z^{k}_i)} \int_{B_{r_k}(z^{k}_j)} \mathbb{T}(x,y) dx dy,
\end{align}
where $Q_{k, z^{k}_i}$ contains the centers that are closed to $z^{k}_i$, it indicates that
\begin{align*}
Q_{k, z^{k}_i} := \left\{z^{k}_l \in Q_k: \  B_{3r_k/2}(z^{k}_l) \cap B_{3r_k/2}(z^{k}_i) \neq \emptyset\right\}.
\end{align*}
The nice feature here is that, the cardinality of $Q_{k, z^{k}_i}$ is finite and depends only on $n$ and $R_0$, i.e. there exists $C(n,R_0)$ such that $|Q_{k, z^{k}_i}| \le C(n,R_0)$. Moreover, it is easily for us to check that
\begin{align*}
B_{r_k}(z^{k}_j) \subset B_{4r_k}(z^{k}_i), \quad \mbox{ for all } \ z^{k}_j \in  Q_{k, z^{k}_i}.
\end{align*}   
Therefore, we are able to estimate the first term on the right-hand side of~\eqref{est:Om-k} as
\begin{align}\label{est:Qk-1}
 \sum_{z^{k}_i \in Q_k} \sum_{z^{k}_j \in Q_{k, z^{k}_i}} &  \int_{B_{r_k}(z^{k}_i)} \int_{B_{r_k}(z^{k}_j)} \mathbb{T}(x,y) dx dy \notag \\
 & \le C(n,R_0) \sum_{z^{k}_i \in Q_k} \int_{B_{4r_k}(z^{k}_i)} \int_{B_{4r_k}(z^{k}_i)} \mathbb{T}(x,y) dx dy.
\end{align}
Applying~\eqref{eq:diff_est_local} in Theorem~\ref{theo:Avelin}, it enables us to obtain
\begin{align}\nonumber
\int_{B_{4r_k}(z^{k}_i)} & \int_{B_{4r_k}(z^{k}_i)} \mathbb{T}(x,y) dx dy \notag \\
& \le r_k^{\alpha+\beta} \int_{B_{4r_k}(z^{k}_i)} \int_{B_{4r_k}(z^{k}_i)}  \frac{|\mathcal{A}(\nabla u(x)) - \mathcal{A}(\nabla u(y))|}{|x-y|^{n+\sigma}} dx dy \notag \\ \label{est:Qk-2}
& \le C(n,p,c_\mathcal{A},\sigma)r_k^{\alpha+\beta-\sigma} \left( \int_{B_{8r_k}(z^{k}_i)} |\mathcal{A}(\nabla u(x))| dx +  r_k \left[|\mu|(B_{8r_k})\right] \right).
\end{align}
Combining between~\eqref{est:Qk-1} and~\eqref{est:Qk-2} together, one gets
\begin{align}\nonumber
& \sum_{z^{k}_i \in Q_k} \sum_{z^{k}_j \in Q_{k, z^{k}_i}}  \int_{B_{r_k}(z^{k}_i)} \int_{B_{r_k}(z^{k}_j)} \mathbb{T}(x,y) dx dy  \\ \label{est:Qk-3}
& \le C(n,p,c_\mathcal{A},\sigma,R_0) r_k^{\alpha+\beta-\sigma}   \left(\sum_{z^{k}_i \in Q_k} \int_{B_{8r_k}(z^{k}_i)} |\mathcal{A}(\nabla u(x))| dx +  r_k \sum_{z^{k}_i \in Q_k} \left[|\mu|(B_{8r_k})\right] \right). 
\end{align}
In the course of the proof, we note that there is a constant $C=C(n)>0$ such that
\begin{align*}
\sum_{z^{k}_i \in Q_k} \chi_{B_{8r_k}(z^{k}_i)}(\xi) \le C \chi_{\Omega_0}(\xi), \quad \forall \xi \in \Omega,
\end{align*}
then for any $f \in L_{\mathrm{loc}}^1(\mathbb{R}^n)$, there holds
\begin{align} \label{eq:property}
\sum_{z^{k}_i \in Q_k} \int_{B_{8r_k}(z^{k}_i)} f(\xi) d\xi  = \sum_{z^{k}_i \in Q_k} \int_{\mathbb{R}^n}\chi_{B_{8r_k}(z^{k}_i)}(\xi) f(\xi) d\xi 
  \le C \int_{\Omega_0} f(\xi) d\xi.
\end{align}
Applying~\eqref{eq:property} to~\eqref{est:Qk-3}, one concludes that
\begin{align}\nonumber
\sum_{z^{k}_i \in Q_k} & \sum_{z^{k}_j \in Q_{k, z^{k}_i}}  \int_{B_{r_k}(z^{k}_i)} \int_{B_{r_k}(z^{k}_j)} \mathbb{T}(x,y) dx dy  \\ \label{est:Qk}
& \le C(n,p,c_\mathcal{A},\sigma,R_0) r_k^{\alpha+\beta-\sigma}   \left(\int_{\Omega_0} |\mathcal{A}(\nabla u(x))| dx +  r_k  \left[|\mu|(\Omega_0)\right] \right). 
\end{align}
On the other hand, for any $x \in B_{r_k}(z^{k}_i)$ and $y \in B_{r_k}(z^{k}_j)$ with $z^{k}_i \in Q_k$, $z^{k}_j \in  Q_k \setminus Q_{k, z^{k}_i}$, we have $|x - y| \ge  r_k$. Applying a similar argument as in the previous inequality~\eqref{est:basic}, one also has
\begin{align*}
\sum_{z^{k}_j \in Q_k \setminus Q_{k, z^{k}_i}} & \int_{B_{r_k}(z^{k}_i)} \int_{B_{r_k}(z^{k}_j)} d^{\alpha}(x) d^{\beta}(y) \frac{|\mathcal{A}(\nabla u(x))|}{|x-y|^{n+\sigma}} dx dy \\
 & \le r_k^{\alpha+\beta}  \int_{B_{r_k}(z^{k}_i)} \left(\sum_{z^{k}_j \in Q_k \setminus Q_{k, z^{k}_i}} \int_{B_{r_k}(z^{k}_j)}  \frac{1}{|x-y|^{n+\sigma}} dy \right) |\mathcal{A}(\nabla u(x))| dx  \\
 & \le r_k^{\alpha+\beta-\sigma}  \int_{B_{r_k}(z^{k}_i)}   \left(\int_{\{|\xi|\ge 1\}}  \frac{1}{|\xi|^{n+\sigma}} d\xi \right) |\mathcal{A}(\nabla u(x))| dx  \\
 & \le C(n,\sigma) r_k^{\alpha+\beta-\sigma}\int_{B_{r_k}(z^{k}_i)} |\mathcal{A}(\nabla u(x))| dx.
\end{align*}
Taking into account the above inequality, we may estimate the last term in~\eqref{est:Om-k} as 
\begin{align}\nonumber
 \sum_{z^{k}_i \in Q_k} \sum_{z^{k}_j \in Q_k \setminus Q_{k, z^{k}_i}} &  \int_{B_{r_k}(z^{k}_i)} \int_{B_{r_k}(z^{k}_j)} \mathbb{T}(x,y) dx dy \\ \nonumber
& \le C(n,\sigma){r_k^{\alpha+\beta-\sigma}} \sum_{z^{k}_i \in Q_k} \int_{B_{r_k}(z^{k}_i)} |\mathcal{A}(\nabla u(x))| dx \\ \label{est:Qk-c}
& \le C(n,\sigma,R_0){r_k^{\alpha+\beta-\sigma}} \int_{\Omega_0} |\mathcal{A}(\nabla u(x))| dx.
\end{align}
Substituting~\eqref{est:Qk} and~\eqref{est:Qk-c} into~\eqref{est:Om-k}, one gets that
\begin{align} \nonumber
(\mathbb{I})_3 & = \sum_{k=1}^{\infty} \int_{\Omega_{k}}\int_{\Omega_k} \mathbb{T}(x,y) dx dy \\  \nonumber
& \le C(n,p,c_\mathcal{A},\sigma) \left( \int_{\Omega_0} |\mathcal{A}(\nabla u(x))| dx \sum_{k=1}^{\infty} r_k^{\alpha+\beta-\sigma}  +   |\mu|(\Omega_0)\sum_{k=1}^{\infty} r_k^{\alpha+\beta-\sigma+1}  \right) \\ \label{est:I33}
& \le C(n,p,c_\mathcal{A},\sigma,R_0) \left( \int_{\Omega_0} |\mathcal{A}(\nabla u(x))| dx    +   |\mu|(\Omega_0)  \right) \sum_{k=1}^{\infty} r_k^{\alpha+\beta-\sigma}. 
\end{align}
We next estimate the last term $(\mathbb{I})_2$ with notice that 
\begin{align}\nonumber
(\mathbb{I})_2 = \sum_{|k-j| = 1} \int_{\Omega_{k}}\int_{\Omega_j} \mathbb{T}(x,y) dx dy & = 2 \sum_{k = 1}^{\infty} \int_{\Omega_{k}}\int_{\Omega_{k+1}} \mathbb{T}(x,y) dx dy \\ \nonumber
& \le 2 \sum_{k = 1}^{\infty} \int_{P_{k}}\int_{P_{k}} \mathbb{T}(x,y) dx dy,
\end{align}
where the new set $P_k$ is defined by
\begin{align*}
P_k := \Omega_k \cup \Omega_{k+1} = \left\{x \in \Omega: \ \frac{r_k}{4} < d(x) \le r_k\right\}.
\end{align*}
In a similar fashion, for $(\mathbb{I})_3$ we may decompose $P_k$ by the same the method to $\Omega_k$ in~\eqref{est:Om-k} and preform the same computation to observe that
\begin{align} \label{est:I32}
(\mathbb{I})_2 \le C(n,p,c_\mathcal{A},\sigma,R_0) \left( \int_{\Omega_0} |\mathcal{A}(\nabla u(x))| dx + |\mu|(\Omega_0) \right) \sum_{k=1}^{\infty} r_k^{\alpha+\beta-\sigma}.
\end{align}
Collecting~\eqref{est:I3-a},~\eqref{est:I31},~\eqref{est:I33} and~\eqref{est:I32}, one can conclude that
\begin{align}\label{est:I3}
(\mathbb{I}) \le C(n,p,c_\mathcal{A},\sigma,\alpha,\beta,R_0)  \left( \int_{\Omega} |\mathcal{A}(\nabla u(x))| dx + |\mu|(\Omega)  \right) \sum_{k=1}^{\infty} r_k^{\alpha+\beta-\sigma}.
\end{align}
Finally, the assumption $\alpha+\beta>\sigma$ allows us to find 
\begin{align*}
\sum_{k=1}^{\infty} r_k^{\alpha+\beta-\sigma} = C R_0^{\alpha+\beta-\sigma}, \quad \mbox{with} \ C = \sum_{k=1}^{\infty} \left(\frac{1}{2}\right)^{(\alpha+\beta-\sigma)k} < \infty,
\end{align*}
which leads to the desired result~\eqref{eq:est-main} from~\eqref{est:I3}.
\end{proof}

\begin{proof}[Proof of Theorem~\ref{theo:main-B}]
For every $\xi, \zeta \in \mathbb{R}^n$, let us recall two following elementary inequalities
\begin{align}\label{est:B-01}
\big|\mathbb{E}(\xi)-\mathbb{E}(\zeta)\big|^{\frac{1}{\gamma}-1} & = \big|\left(|\xi| + \kappa\right)^{\gamma p - \gamma-1} \xi - \left(|\zeta| + \kappa\right)^{\gamma p - \gamma-1} \zeta\big|^{\frac{1}{\gamma}-1} \notag \\
& \le \left(\left(|\xi| + \kappa\right)^{\gamma p - \gamma} + \left(|\zeta| + \kappa\right)^{\gamma p - \gamma}\right)^{\frac{1}{\gamma}-1} \notag \\ 
& \le C(p,\gamma) \left(|\xi| + |\zeta| + \kappa\right)^{(p-1)(1-\gamma)},
\end{align}
and 
\begin{align}\label{est:B-02}
\big|\mathbb{E}(\xi)-\mathbb{E}(\zeta)\big|^{\frac{1}{\gamma}} & = \big|\left(|\xi| + \kappa\right)^{\gamma p - \gamma-1} \xi - \left(|\zeta| + \kappa\right)^{\gamma p - \gamma-1} \zeta\big| \notag \\
& \le C(p,\gamma) \left(|\xi| + |\zeta| + \kappa\right)^{\gamma p - \gamma-1} |\xi - \zeta|.
\end{align}
Combining~\eqref{est:B-01} and~\eqref{est:B-02}, it leads to
\begin{align*}
|\mathbb{E}(\xi)-\mathbb{E}(\zeta)|^{\frac{1}{\gamma}} & \le C(p,\gamma) \left(|\xi| + |\zeta|+ \kappa\right)^{p-2} |\xi - \zeta|,
\end{align*}
and together with assumption~\eqref{eq:cond-ellipgrow}, it allows us to arrive
\begin{align}\label{est:B-2}
|\mathbb{E}(\xi)-\mathbb{E}(\zeta)|^{\frac{1}{\gamma}} & \le C(p,\gamma) |\mathcal{A}(\xi) - \mathcal{A}(\zeta)|.
\end{align}
On the other hand, thanks to~\eqref{est:B-2} and Theorem~\ref{theo:main}, one has 
\begin{align*}
[\mathbb{E}(\nabla u)]_{W_G^{\gamma\sigma,\frac{1}{\gamma}}(\Omega;\alpha,\beta)} & = \left[\int_{\Omega} \int_{\Omega} d^{\alpha}(x) d^{\beta}(y) \frac{|\mathbb{E}(\nabla u(x))-\mathbb{E}(\nabla u(y))|^{\frac{1}{\gamma}}}{|x-y|^{n+\sigma}} dx dy\right]^{\gamma} \notag \\
& \le C(p,\gamma)\left[\int_{\Omega} \int_{\Omega} d^{\alpha}(x) d^{\beta}(y) \frac{|\mathcal{A}(\nabla u(x)) - \mathcal{A}(\nabla u(y))|}{|x-y|^{n+\sigma}} dx dy\right]^{\gamma} \notag \\
& \le C(c_{\mathcal{A}},\sigma,n,p,\alpha,\beta,\gamma) \left(\int_{\Omega} |\mathcal{A}(\nabla u(x))| dx + |\mu|(\Omega)\right)^{\gamma},
\end{align*}
which completes the proof.
\end{proof}


\begin{thebibliography}{99}
\footnotesize

\bibitem{AKM18} B. Avelin, T. Kuusi, G. Mingione, {\em Nonlinear Calder{\'o}n-Zygmund theory in the limiting case}, Arch. Rational. Mech. Anal., {\bf 227} (2018), 663--714.

\bibitem{BDW2020} A. Kh. Balci, L. Diening, M. Weimar, {\em Higher order Calder{\'o}n-Zygmund estimates for the $p$-Laplace equation}, J. Differential Equations, {\bf 268} (2020), 590--635.

\bibitem{BBGGPV1995} P. Benilan P., L. Boccardo, T. Gallou\"et, R. Gariepy R., M. Pierre, J. L. Vazquez J. L., {\em An $L^1$-theory of existence and uniqueness of solutions of nonlinear elliptic equations}, Ann. Scuola Norm. Sup. Pisa Cl. Sci. (IV), {\bf 22} (1995), 241--273.

\bibitem{BCDK2018} D. Breit, A. Cianchi, L. Diening, T. Kuusi and S. Schwarzacher, {\em Pointwise Calderon-Zygmund gradient estimates for the $p$-Laplace system}, J. Math. Pures Appl. (9), {\bf 114} (2018), 146--190.

\bibitem{BG1989} L. Boccardo, T. Gallou\"et, {\em Nonlinear elliptic and parabolic equations involving measure data}, J. Funct. Anal., {\bf 87} (1989), 149--169.

\bibitem{BG1992} L. Boccardo, T. Gallou\"et, {\em Nonlinear elliptic equations with right-hand side measures}, Comm. Partial Differential Equations, {\bf 17} (1992), 641--655.

\bibitem{CM2014} A Cianchi, VG Maz'ya, {\em Gradient regularity via rearrangements for $p$-Laplacian type elliptic boundary value problems}, J. Eur. Math. Soc. (JEMS), {\bf 16}(3) (2014), 571--595.

\bibitem{Dall1996} A. Dall'Aglio, {\em Approximated solutions of equations with $L^1$-data. Application to the $H$-convergence of quasi-linear parabolic equations}, Ann. Mat. Pura Appl., {\bf 170}(4) (1996), 207--240.

\bibitem{Maso1997} G. Dal Maso, F. Murat, L. Orsina, A. Prignet, {\em Definition and existence of renormalized solutions of elliptic equations with general measure data}, Comptes Rendus de l’Acad\'emie Des Sciences - Series I - Mathematics, {\bf 325}(5) (1997), 481--486.

\bibitem{Maso1999} G. Dal Maso, F. Murat, L. Orsina, A. Prignet, {\em Renormalized solutions for elliptic equations with general measure data}, Ann. Sc. Norm. Super Pisa Cl. Sci., {\bf 28} (1999), 741--808.

\bibitem{DPV12} E. Di Nezza, G. Palatucci, E. Valdinoci, {\em Hitchhiker’s guide to the fractional Sobolev spaces}, Bull. Sci. Math., {\bf 136} (2012), 521--573.


\bibitem{DHM1997} G. Dolzmann G. N. Hungerb\"uhler,  S. M\"uller, {\em Non-linear elliptic systems with measure valued right hand side}, Math. Z., {\bf 226} (1997), 545--574.

\bibitem{DM2010} F. Duzaar, G. Mingione, {\em Local Lipschitz regularity for degenerate elliptic systems}, Ann. Inst. H. Poincar{\'e} Anal. Non Lin{\'e}aire, {\bf 27}(6) (2010), 1361--1396.

\bibitem{HH2017} P. Harjulehto, P. H\"ast\"o, {\em Riesz potential in generalized Orlicz spaces}, Forum Math., {\bf 29} (2017), 229--244.

\bibitem{KM2018} T. Kuusi, G. Mingione, {\em Vectorial nonlinear potential theory}, J. Europ. Math. Soc. (JEMS), {\bf 20} (2018), 929--1004.

\bibitem{Min2007} G. Mingione, {\em The Calder\'on-Zygmund theory for elliptic problems with measure data}, Ann Scu. Norm. Sup. Pisa Cl. Sci. (V), {\bf 6} (2007), 195--261.

\bibitem{HP2019} Q.-H. Nguyen, N.-C. Phuc, {\em Good-$\lambda$ and Muckenhoupt-Wheeden type bounds in  quasilinear measure datum problems, with applications}, Math. Ann., {\bf 374}(1-2) (2019), 67--98.

\bibitem{Serrin}  J. Serrin, {\em Pathological solutions of elliptic differential equations}, Ann. Scu. Norm. Sup. Pisa Cl. Sci., {\bf 18} (1964), 385--387.

\bibitem{MPT2018} M.-P. Tran, {\em Good-$\lambda$ type bounds of quasilinear elliptic equations for the singular case},  Nonlinear Anal., {\bf 178} (2019), 266--281.

\bibitem{MPT19} M.-P. Tran, T.-N. Nguyen, {\em Global gradient estimates for very singular nonlinear elliptic equations with measure data}, preprint, arXiv:1909.06991, 39 pp.

\bibitem{PNCCM} M.-P. Tran, T.-N. Nguyen, {\em Lorentz-Morrey global bounds for singular quasilinear elliptic equations with measure data}, Commun. Contemp. Math., {\bf 22}(5) (2020), 1950033, 30 pp.

\bibitem{PNJDE} M.-P. Tran, T.-N. Nguyen, {\em New gradient estimates for solutions to quasilinear divergence form elliptic equations with general Dirichlet boundary data}, J. Differential Equations, {\bf 268}(4) (2020), 1427--1462.

\bibitem{Veron} L. V\'eron, {\em Elliptic equations involving measures}, Stationary Partial Differential Equations, vol. I, Handb. Differ. Equ., North-Holland, Amsterdam (2004), 593--712.
\end{thebibliography}
\end{document}